\documentclass[12pt]{article}
\usepackage[authoryear]{natbib}
\usepackage{amsmath,amssymb,amsfonts,amsthm,bbm}
\usepackage{epic,eepic,psfrag,epsfig,longtable}
\newtheorem{thm}{Theorem}

\newtheorem{lemma}[thm]{Lemma}

\newtheorem{prop}[thm]{Proposition}

\numberwithin{equation}{section}

\newenvironment{prooftitle}[1]{{\noindent \textbf{Proof #1.}}\\}

\title{Theoretical properties of the log-concave maximum likelihood estimator of a multidimensional density}
\author{Madeleine Cule and Richard Samworth\thanks{Address for correspondence: Dr Richard Samworth, Statistical Laboratory, Centre for Mathematical Sciences, Wilberforce Road, Cambridge, CB3 0WB, UK.  Email: r.samworth@statslab.cam.ac.uk}\\University of Cambridge, UK}
\date{}

\begin{document}

\maketitle

\begin{abstract}
We present theoretical properties of the log-concave maximum likelihood estimator of a density based on an independent and identically distributed sample in $\mathbb{R}^d$.  Our study covers both the case where the true underlying density is log-concave, and where this model is misspecified.  We begin by showing that for a sequence of log-concave densities, convergence in distribution implies much stronger types of convergence -- in particular, it implies convergence in Hellinger distance and even in certain exponentially weighted total variation norms.  In our main result, we prove the existence and uniqueness of a log-concave density that minimises the Kullback--Leibler divergence from the true density over the class all log-concave densities, and also show that the log-concave maximum likelihood estimator converges almost surely in these exponentially weighted total variation norms to this minimiser.  In the case of a correctly specified model, this demonstrates a strong type of consistency for the estimator; in a misspecified model, it shows that the estimator converges to the log-concave density that is closest in the Kullback--Leibler sense to the true density.   
\end{abstract}

\section{Introduction}

Although work on shape-constrained density estimation dates back to the celebrated paper of \citet{Grenander1956} on the estimation of a decreasing density on the positive half-line, it is in recent years that the area has enjoyed its most significant interest.  This is partly because algorithmic and technological advances now allow the computation of estimators that would not previously have been feasible, and partly because statisticians now have more tools at their disposal for the study of the theoretical properties of these estimators.  

The attraction of the use of these estimators is that, in contrast to alternative nonparametric density estimation methods such as those based on kernels or wavelets, they provide fully automatic procedures, with no smoothing parameters to choose.  Such procedures are particularly desirable when the data are multidimensional, and the choice of (often multiple) smoothing parameters is particularly problematic.

The properties of the Grenander estimator are now quite well understood, thanks to the works of \citet{MarshallProschan1965}, \citet{PrakasaRao1969}, \citet{Devroye1987}, \citet{Birge1989}, \citet{vdG1993} and \citet{BJPSW2009}.  Other examples of shape constraints on univariate densities that have been studied in the literature include convexity \citep{GJW2001,DRW2007} and $k$-monotonicity \citep{BalabdaouiWellner2008}.  It is also known that a maximum likelihood estimator does not exist over the class of unimodal densities -- cf. \citet{Birge1997}.

Log-concavity has become an intensively-studied shape constraint for densities recently -- see, for example, \citet{Walther2002}, \citet{DHR2007}, \citet{PWM2007}, \citet{DumbgenRufibach2009}, \citet{BRW2009}.  The class of univariate log-concave densities includes many common parametric families, such as the normal, $\Gamma(\alpha,\beta)$ ($\alpha \geq 1$), $\mathrm{Beta}(\alpha,\beta)$ ($\alpha, \beta \geq 1$), Weibull($\alpha$) ($\alpha \geq 1$), Gumbel, logistic and Laplace densities; see \citet{BagnoliBergstrom1989} for other examples.  Among the desirable properties enjoyed by the class are the facts that it is closed under convolution (e.g. \citet[Theorem~2.18]{DharmadhikariJoag-Dev1988}) and the taking of pointwise limits.

The existence and uniqueness of a log-concave maximum likelihood estimator $\hat{f}_n$ of a density $f_0$ based on a sample $X_1,\ldots,X_n$ in $\mathbb{R}^d$ was proved in \citet{CSS2007}.  There,  the structure of $\hat{f}_n$ was outlined and an algorithm for its computation was described.  The algorithm was implemented in the \textbf{R} package \verb'LogConcDEAD' \citep{CGS2007,CGS2009}.  

In this paper, we study the statistical properties of the estimator.  Importantly, our results do not assume that the underlying density is log-concave.  To the best of our knowledge, such results have not been obtained before even for univariate data, but are of interest because in practice it is impossible to tell from a sample of data whether the assumption of log-concavity is satisfied.  It is therefore natural to seek assurance that the estimator will behave sensibly if the condition is violated.  In our main result (cf. Theorem~\ref{Thm:Main} below), we prove under very mild conditions the existence and uniqueness of a log-concave density $f^*$ that minimises the Kullback--Leibler divergence from $f_0$ and show that there is an interval of values of $a$ for which 
\[
\int_{\mathbb{R}^d} e^{a\|x\|}|\hat{f}_n(x) - f^*(x)| \stackrel{a.s.}{\rightarrow} 0
\]
as $n \rightarrow \infty$.  Moreover, if $f^*$ is continuous, then $\sup_{x \in \mathbb{R}^d} e^{a\|x\|}|\hat{f}_n(x) - f^*(x)| \stackrel{a.s.}{\rightarrow} 0$ as $n \rightarrow \infty$.  The upper bound for the range of allowable values of $a$ is explicitly linked to the rate of tail decay of $f^*$.  In the case where $f_0$ is log-concave, it is well-known that $f^* = f_0$, so the result demonstrates the strong consistency of $\hat{f}_n$ in exponentially weighted total variation norms, and in exponentially weighted supremum norms if $f_0$ is continuous.  If the true density is not log-concave, we see that the limiting density is the one that is closest (in the Kullback--Leibler sense) to $f_0$.  As described in Section~\ref{Sec:Main} below, this result strengthens what was previously known even for the case $d=1$.  

The rest of this paper is organised as follows.  In Section~\ref{Sec:Convergence}, we study sequences of log-concave densities that converge in distribution to a limiting density, and demonstrate that the convergence must also occur in much stronger senses.  In Section~\ref{Sec:Main}, we show that, with probability one, the estimator is uniformly bounded above on $\mathbb{R}^d$, and uniformly bounded below on compact subsets in the interior of the support of the true density.  This enables us to state and prove the main result described in the previous paragraph.  Further auxiliary results can be found in the Appendix.

\section{Convergence of log-concave densities}
\label{Sec:Convergence}

We begin with an elementary lemma, whose proof is given in the Appendix.  Let $f$ be a log-concave density on $\mathbb{R}^d$.  
\begin{lemma}
\label{Lemma:ExpBound}
There exist $a > 0$ and $b \in \mathbb{R}$ such that $f(x) \leq e^{-a\|x\| + b}$ for all $x \in \mathbb{R}^d$.
\end{lemma}
Notice in particular that if $X$ has density $f$, then Lemma~\ref{Lemma:ExpBound} implies that the moment generating function of $X$ is finite in an open neighbourhood of the origin.

If $f,f_1,f_2,\ldots$ are densities on $\mathbb{R}^d$, we write $f_n \stackrel{d}{\rightarrow} f$ for the convergence in distribution of the corresponding measures; in other words, $f_n \stackrel{d}{\rightarrow} f$ means $\int g(x)f_n(x) \, dx \rightarrow \int g(x)f(x) \, dx$ for all bounded, continuous functions $g:\mathbb{R}^d \rightarrow \mathbb{R}$.  The following result shows that when it is known that a sequence of densities is log-concave, convergence in distribution in fact implies much stronger forms of convergence.  A similar result, proved independently at around the same time and using different techniques, can be found in \citet{SHD2009}.  We write $\mu$ for Lebesgue measure on $\mathbb{R}^d$.
\begin{prop}
\label{Prop:Convergence}
Let $(f_n)$ be a sequence of log-concave densities on $\mathbb{R}^d$ with $f_n \stackrel{d}{\rightarrow} f$ for some density $f$.  Then:
\renewcommand{\theenumi}{\alph{enumi}}
\renewcommand{\labelenumi}{(\theenumi)}
\begin{enumerate}
\item $f$ is log-concave
\item $f_n \rightarrow f$, $\mu$-almost everywhere
\item Let $a_0 > 0$ and $b_0 \in \mathbb{R}$ be such that $f(x) \leq e^{-a_0\|x\| + b_0}$.  Then for every $a < a_0$, we have $\int_{\mathbb{R}^d} e^{a\|x\|}|f_n(x) - f(x)| \, dx \rightarrow 0$ and, if $f$ is continuous, $\sup_{x \in \mathbb{R}^d} e^{a\|x\|}|f_n(x) - f(x)| \rightarrow 0$.
\end{enumerate}
\end{prop} 
\begin{proof}
(a) This part of the proposition can be deduced from Theorems~2.8 and~2.10 of \citet{DharmadhikariJoag-Dev1988}. Their proof relies on a non-trivial correspondence between log-concave probability measures and log-concave densities, which in turn depends on several other facts about log-concavity -- cf. \citet[][pp.46--54]{DharmadhikariJoag-Dev1988}.  We give an alternative proof because it is perhaps a little more direct, and because it forms part of the proof of part~(b) below.

Let $f_n \stackrel{d}{\rightarrow} f$.  Our proof relies on two crucial results.  The first is that if $\mathcal{D}$ is the class of all Borel-measurable, convex subsets of $\mathbb{R}^d$, then 
\[
\sup_{D \in \mathcal{D}} \biggl|\int_D (f_n - f)\biggr| \rightarrow 0
\]
as $n \rightarrow \infty$ \citep[Theorem~1.11, p.22]{BhattacharyaRao1976}.  The second is Lebesgue's differentiation theorem: recall that a family $\{A_\delta:\delta > 0\}$ of Borel subsets of $\mathbb{R}^d$ \emph{shrinks nicely} to $x_0 \in \mathbb{R}^d$ with eccentricity bound $\eta > 0$ if $A_\delta \subseteq B_\delta$, where $B_\delta$ is the closed (Euclidean) ball of radius $\delta$ centered at $x_0$, and if the family is such that $\mu(A_\delta) > \eta \mu(B_\delta)$ for all $\delta > 0$.  Then for $\mu$-almost all $x_0$, we have
\begin{equation}
\label{Eq:Leb}
\frac{1}{\mu(A_\delta)} \int_{A_\delta} |f(x) - f(x_0)| \, dx \rightarrow 0
\end{equation}
as $\delta \rightarrow 0$, for every family $\{A_\delta:\delta > 0\}$ that shrinks nicely to $x_0$ \citep[][Theorem~3.21]{Folland1999}.  Points $x_0$ at which this equality holds are called \emph{Lebesgue points} of $f$; notice that any continuity point of $f$ is certainly a Lebesgue point of $f$.  In fact, it will be helpful to note the following small generalisation: if we have a sequence $\{A_{k,\delta}:k \in \mathbb{N},\delta > 0\}$ of families that shrink nicely to $x_0$ with the same eccentricity bound $\eta$, then the convergence in~(\ref{Eq:Leb}) is uniform in $k$. 

Consider $E_1 = \{x \in \mathbb{R}^d: \liminf f_n(x) < f(x)\}$, and suppose for a contradiction that $\mu(E_1) > 0$.  Then there exists a Lebesgue point $x_0$ of $f$ satisfying $x_0 \in E_1$ and $f(x_0) < \infty$.  Letting $\epsilon = f(x_0) - \liminf f_n(x_0)$, find a subsequence $(f_{n_k})$ with $f_{n_k}(x_0) < f(x_0) - 3\epsilon/4$, and for $\delta > 0$, define the convex sets
\[
D_{k,\delta} = \{x \in B_\delta: f_{n_k}(x) \geq f(x_0) - \epsilon/2\}.
\]
Observe by the concavity of $\log f_{n_k}$ that if $x \in D_{k,\delta}$ then $2x_0 - x \in B_{\delta} \setminus D_{k,\delta}$.  It follows that $\mu(B_\delta \setminus D_{k,\delta}) \geq \mu(B_{\delta})/2$.  This means that we can apply Lebesgue's differentiation theorem to choose $\delta > 0$ small enough that for every $k$,
\[
\frac{1}{\mu(B_\delta \setminus D_{k,\delta})} \biggl|\int_{B_\delta \setminus D_{k,\delta}} \{f(x_0) - f(x)\} \, dx \biggr| \leq \frac{\epsilon}{8}.
\]
But then
\begin{align*}
\int_{B_{\delta}} (f - f_{n_k}) &= \int_{B_\delta \setminus D_{k,\delta}} \! \! \{f(x) - f(x_0) + f(x_0) - f_{n_k}(x)\} \, dx + \int_{D_{k,\delta}} (f - f_{n_k}) \\
&\geq -\frac{\epsilon}{8}\mu(B_{\delta}) + \frac{\epsilon}{4}\mu(B_{\delta}) - \sup_{D \in \mathcal{D}} \biggl|\int_{D} (f - f_{n_k})\biggr|.
\end{align*}
We conclude that $\liminf_k \int_{B_{\delta}} (f - f_{n_k}) \geq \epsilon\mu(B_{\delta})/4$, a contradiction.  Hence $\mu(E_1) = 0$.  

Thus, without loss of generality, we may assume $f \leq \liminf_n f_n$. But by Fatou's lemma,
\[
1 = \int f \leq \int \liminf f_n \leq \liminf \int f_n = 1,
\]
so in fact we may assume $f = \liminf f_n$.  Since $\liminf f_n$ is log-concave, this proves~(a).

(b) Now suppose that $\mu(E_2) > 0$, where $E_2 = \{x \in \mathbb{R}^d: \limsup f_n(x) > f(x)\}$.   Since $f$ is log-concave and so continuous almost everywhere, let $x_0 \in E_2$ be continuity point of $f$.  Define $\epsilon_0 \in (0,\infty]$ by $\epsilon_0 =  \limsup f_n(x_0) - f(x_0)$ and set $\epsilon = \min(1,\epsilon_0)$.  We can find a subsequence $(f_{n_k})$ with $f_{n_k}(x_0) > f(x_0) + 3\epsilon/4$ for all $k$.  Define the convex sets
\[
\tilde{D}_{k,\delta} = \{x \in B_\delta: f_{n_k}(x) \geq f(x_0) + \epsilon/2\}.
\]
There are three cases to consider:

\emph{Case (i)}: Suppose the sequence $\{\tilde{D}_{k,\delta}:k \in \mathbb{N},\delta > 0\}$ of families shrink nicely to $x_0$ with the same eccentricity bound $\eta$.  Find $\delta > 0$ such that $|f(x) - f(x_0)| \leq \epsilon/4$ for all $x \in B_{\delta}$.  Then, for every $k$,
\[
\int_{\tilde{D}_{k,\delta}} (f_{n_k} - f) \geq \frac{\mu(\tilde{D}_{k,\delta})\,\epsilon}{4} \geq \frac{\mu(B_{\delta})\,\eta\,\epsilon}{4},
\]
contradicting our first crucial result.

\emph{Case (ii)}: Suppose we are not in Case~(i), and that $f(x_0) > 0$, so that by reducing $\epsilon$ if necessary we may assume $f(x_0) > \epsilon/2$.  In this case, since for each $k$ the ratio $\mu(\tilde{D}_{k,\delta})/\mu(B_\delta)$ is decreasing as $\delta$ increases, there exist $\delta > 0$ and positive integers $k(1) < k(2) < \ldots$ such that 
\[
\frac{\mu(\tilde{D}_{k(l),\delta})}{\mu(B_\delta)} \leq \frac{t^d}{2},
\]
where 
\[
t = \frac{\log \bigl(f(x_0) + 3\epsilon/4\bigr) - \log \bigl(f(x_0) + \epsilon/2\bigr)}{\log \bigl(f(x_0) + 3\epsilon/4\bigr) - \log \bigl(f(x_0) - \epsilon/2\bigr)}.
\]
It is straightforward to show, using the concavity of $\log f_{n_k}$, that $\mu(D_{k,\delta}) \leq \mu(\tilde{D}_{k,\delta})/t^d$, where as above,
\[
D_{k,\delta} = \{x \in B_\delta: f_{n_k}(x) \geq f(x_0) - \epsilon/2\}.
\]
We may also assume that $|f(x) - f(x_0)| \leq \epsilon/4$ for all $x \in B_{\delta}$.  We conclude that for all $l$,
\begin{align*}
\int_{B_\delta} (f_{n_{k(l)}} - f) &= \int_{B_\delta \setminus \tilde{D}_{k(l),\delta}}  (f_{n_{k(l)}} - f) + \int_{D_{k(l),\delta}}  (f_{n_{k(l)}} - f) \\
&\leq -\frac{\epsilon}{4}\mu(B_\delta \setminus D_{k(l),\delta}) + \sup_{D \in \mathcal{D}} \int_D (f_{n_{k(l)}} - f),
\end{align*}
so $\limsup_l \int_{B_\delta} (f_{n_{k(l)}} - f) \leq -\frac{\epsilon}{8}\mu(B_\delta)$, a contradiction.  

\emph{Case (iii)}: Finally, if $f(x_0) = 0$, then without loss of generality we can find $\delta > 0$ such that $f(x) = 0$ for all $x \in B_\delta$.  For each $t >0$, we have $\mu(\{x \in B_\delta:f_{n_k}(x) \geq t\}) \rightarrow 0$ as $k \rightarrow \infty$.  Choose $t > 0$ small enough that $f_{n_k}(x_0) \geq t$ and such that there exist points $x_1,\ldots,x_d$ in the interior of the effective domain of $\log f$, denoted $\mathrm{int}(\mathrm{dom} \ \log f)$, with $f(x_j) \geq 2t$ and $\mu\bigl(\mathrm{conv}\{x_0,x_1,\ldots,x_d\}\bigr) > 0$, where $\mathrm{conv}\{x_0,x_1,\ldots,x_d\}$ denotes the convex hull of $\{x_0,x_1,\ldots,x_d\}$.  As we cannot have $\mathrm{conv}\{x_0,x_1,\ldots,x_d\}$ contained in $\{x \in \mathbb{R}^d: f_{n_k}(x) \geq t\}$ eventually, there exists $j \in \{1,\ldots,d\}$ and a further subsequence $(f_{n_{k(l)}})$ such that $f_{n_{k(l)}}(x_j) < t$.  We then obtain a contradiction as in the proof of Proposition~\ref{Prop:Convergence}(a).  Hence $\mu(E_2) = 0$, as required.  This proves~(b).  

(c) Write $\phi_n = \log f_n$ and $\phi = \log f$.  Without loss of generality assume $0 \in \mathrm{int}(\mathrm{dom} \ \phi)$, and let $\eta > 0$ be small enough that $B_\eta(0)$, the closed ball of radius $\eta > 0$ about the origin, is contained in $\mathrm{int}(\mathrm{dom} \ \phi)$.  

Fix $a \in (0,a_0)$, and let $\delta = a_0 - a$.  By Lemma~\ref{Lemma:ExpBound}, we can find $R > 0$ such that $\frac{1}{\|x\|}\{\phi(x) - \phi(0)\} \leq -\bigl(a + \frac{3\delta}{4}\bigr)$ for all $\|x\| \geq R/2$.  We claim there exists $n_0$ such that
\begin{equation}
\label{Eq:Claim}
\frac{\phi_n(x) - \phi_n(0)}{\|x\|} \leq -\Bigl(a + \frac{\delta}{4}\Bigr)
\end{equation}
for all $\|x\| \geq R$ and $n \geq n_0$.  Note that since, for each $n$, the ratio on the left-hand side of~(\ref{Eq:Claim}) is a decreasing function of $\|x\|$, it suffices to prove that the inequality in~(\ref{Eq:Claim}) holds for all $\|x\|=R$ and $n \geq n_0$.  This is straightforward to see if the ball of radius $R$ about the origin is in $\mathrm{int}(\mathrm{dom} \ \phi)$, because in that case $\phi_n \rightarrow \phi$ uniformly on this ball \citep[][Theorem~10.8]{Rockafellar1997}.  In general, however, we argue as follows.  Suppose we can find a subsequence $(\phi_{n_k})$ and a sequence $(x_k)$ with $\|x_k\| = R$ such that 
\[
\frac{\phi_{n_k}(x_k) - \phi_{n_k}(0)}{\|x_k\|} > -\Bigl(a + \frac{\delta}{4}\Bigr)
\]
for all $k$.  Let $C_k = A_k \cap B_k$, where $A_k = \{\lambda x_k + (1-\lambda)y:\lambda \in [0,1], y \in B_\eta(0)\}$ and $B_k = \{y \in \mathbb{R}^d: \|y - x_k\| \leq R/2\}$, so that $C_k$ is convex and $\mu(C_k) = \zeta > 0$, independent of $k$.  By reducing $\eta > 0$ if necessary, we may assume $\frac{R/3}{R/3  - \eta} \leq 1 + \frac{\delta}{8(a+\delta/4)}$ and $\eta < R/4$.  Finally, since $(\phi_{n_k})$ is equi-Lipschitzian on $B_\eta(0)$ \citep[][Theorem~10.8]{Rockafellar1997}, we may assume $\eta$ is small enough that $\phi_{n_k}(y) \geq \phi_{n_k}(0) - \frac{\delta R}{16}$ for all $y \in B_{\eta}(0)$.  Since any $x^* = \lambda x_k + (1-\lambda)y \in C_k$ has $\lambda R - \eta \leq \|x^*\| \leq R$ and $\lambda \geq 1/3$, we have
\begin{align*}
\frac{\phi_{n_k}(x^*) - \phi_{n_k}(0)}{\|x^*\|} &\geq \frac{\lambda\phi_{n_k}(x_k) + (1-\lambda) \phi_{n_k}(y) - \phi_{n_k}(0)}{\|x^*\|} \\
&\geq \frac{\lambda\phi_{n_k}(x_k) -\lambda \phi_{n_k}(0)}{\|x^*\|} - \frac{\delta}{8} \\
&\geq \frac{-\lambda(a+\frac{\delta}{4})R}{\|x^*\|} - \frac{\delta}{8} \geq -\Bigl(a + \frac{\delta}{2}\Bigr).
\end{align*}
From this we deduce that
\[
\liminf_{k \rightarrow \infty} \int_{C_k} (f_{n_k} - f) \geq \zeta\{e^{-\frac{1}{2}(a+\delta/2)R + \phi(0)} - e^{-\frac{1}{2}(a+3\delta/4)R + \phi(0)}\},
\]
contradicting the first crucial result in the proof of Proposition~\ref{Prop:Convergence}(a).  This establishes our claim at~(\ref{Eq:Claim}).  But this means there exists $b \in \mathbb{R}$ such that $\sup_{n \geq n_0} f_n(x) \leq e^{-(a + \delta/4)\|x\| +b}$.  From Proposition~\ref{Prop:Convergence}(b) and the dominated convergence theorem we conclude that 
\[
\int_{\mathbb{R}^d} e^{a\|x\|}|f_n(x) - f(x)| \, dx \rightarrow 0.
\]
Now suppose that $f$ is continuous and let $\epsilon \in (0,1)$.  Choose $R > 0$ large enough that $f(x) + \sup_{n \geq n_0} f_n(x) \leq \epsilon e^{-a\|x\|}/2$ for all $\|x\| \geq R$.  Then certainly,
\[
\sup_{\|x\| \geq R} e^{a\|x\|}|f_n(x) - f(x)| \leq \epsilon
\]
for $n \geq n_0$.  Using the continuity of $f$, we can choose a closed, convex set $S \subseteq \mathrm{int}(\mathrm{dom} \ \phi) \cap B_R(0)$ such that $f(x) < e^{-aR}/2$ for all $x \in S^c$.  Since $f_n \rightarrow f$ uniformly on $S$, we may assume $\sup_{x \in S} |f_n(x) - f(x)| \leq \epsilon e^{-aR}/2$ for all $n \geq n_0$.  Finally, for $x \in B_R(0) \setminus S$, we may assume $\epsilon > 0$ is small enough that $f_n(0) \geq \epsilon e^{-aR}$ for $n \geq n_0$.  Since $f_n(x) \leq \epsilon e^{-aR}$ for $x$ on the boundary of $S$ and $n \geq n_0$, it follows that $f_n(x) \leq \epsilon e^{-aR}$ for $x \in B_R(0) \setminus S$ and $n \geq n_0$.  We deduce that 
\[
\sup_{x \in \mathbb{R}^d} e^{a\|x\|}|f_n(x) - f(x)| \leq \epsilon
\]
for all $n \geq n_0$.
\end{proof}

\section{Theoretical properties of the log-concave maximum likelihood estimator}
\label{Sec:Main}

Let $X_1,X_2,\ldots$ be an independent and identically distributed sequence with density $f_0$ (not necessarily log-concave), and for $n \geq d+1$ let $\hat{f}_n$ denote the log-concave maximum likelihood estimator of $f_0$ based on $X_1,\ldots,X_n$.  Let $E$ denote the \emph{support} of $f_0$; that is, the smallest closed set with $\int_E f_0 = 1$.  The lemma below establishes appropriate upper and lower bounds for $\hat{f}_n$.  Part~(a) of the lemma strengthens Theorem~3.2 of \citet{PWM2007}, where for the case of univariate data and a log-concave underlying density it was shown that the random variable $\sup_{n \in \mathbb{N}} \sup_{x \in \mathbb{R}^d} \hat{f}_n(x)$ is finite with probability one.  To the best of our knowledge, lower bounds such as that appearing in part~(b) have not been studied before.  
\begin{lemma}
\label{Lemma:Bound}
Suppose that $\int_{\mathbb{R}^d} \|x\|f_0(x) \, dx < \infty$.
\renewcommand{\theenumi}{\alph{enumi}}
\renewcommand{\labelenumi}{(\theenumi)}
\begin{enumerate}
\item There exists a constant $C > 0$ such that, with probability one,
\[
\limsup_{n \rightarrow \infty} \sup_{x \in \mathbb{R}^d} \hat{f}_n(x) \leq C.
\]
\item Let $S$ be a compact subset of $\mathrm{int}(E)$.  There exists a constant $c > 0$ such that, with probability one,
\[
\liminf_{n \rightarrow \infty} \inf_{x \in \mathrm{conv} \, S} \hat{f}_n(x) \geq c.
\] 
\end{enumerate}
\end{lemma}
\begin{proof}
(a) Let $g(x) = \exp(-\|x\|+b)$, where the normalisation constant $b$ is chosen to ensure $g$ is a density, so that 
\[
\int f_0 \log g = -\int \|x\| f_0(x) \, dx + b \equiv k + 1,
\]
say.  Now let $C = e^M$, where $M$ is large enough that $M > k+1$ and such that $\int_D f_0 \leq 1/4$ whenever $\mu(D) \leq 2^{d+3}(M - k)^d e^{-M}$.  Let $f$ be any log-concave density with $\sup_{x \in \mathbb{R}^d} f(x) = C$.  We claim that, for sufficiently large $n$, the log-concave density $g$ has higher log-likelihood.  More precisely, if `i.o.' stands for `infinitely often', we claim that
\[
\mathbb{P}\biggl(\frac{1}{n}\sum_{i=1}^n \log f(X_i) > \frac{1}{n}\sum_{i=1}^n \log g(X_i) \ \mathrm{i.o.}\biggr) = 0.
\]
The result follows immediately from this claim.  To establish the claim, write $\phi = \log f$, and observe that
\begin{align}
\label{Eq:i.o.}
\mathbb{P}\biggl(\frac{1}{n}&\sum_{i=1}^n \phi(X_i) > \frac{1}{n}\sum_{i=1}^n \log g(X_i) \ \mathrm{i.o.}\biggr) \\
&\leq \mathbb{P}\biggl(\biggl|\frac{1}{n}\sum_{i=1}^n \log g(X_i) - (k+1)\biggr| > 1 \ \mathrm{i.o.}\biggr) + \mathbb{P}\biggl(\frac{1}{n}\sum_{i=1}^n \phi(X_i) > k \ \mathrm{i.o.}\biggr). \nonumber
\end{align}
The first term on the right-hand side of~(\ref{Eq:i.o.}) is zero, by the strong law of large numbers.

To prove the second term on the right-hand side of~(\ref{Eq:i.o.}) is zero, the idea is to show that there is only a small set on which $\phi$ is large, so with high probability only a small proportion of the observations are in this set.  To this end, let $D_t = \{x \in \mathbb{R}^d: \phi(x) \geq t\}$.  By concavity of $\phi$, for $t \in [2k-M,M]$, we have
\[
\mu(D_t) \geq \Bigl(\frac{M-t}{2M-2k}\Bigr)^d \mu(D_{2k-M}).
\]
It follows by Fubini's theorem that
\begin{align*}
1 &\geq \int_{\mathbb{R}^d} f \mathbbm{1}_{\{\log f \geq 2k - M\}} = \int_{\mathbb{R}^d} \int_0^{e^M} \mathbbm{1}_{\{t \leq f(x)\}} \, dt \mathbbm{1}_{\{\log f(x) \geq 2k - M\}} \, dx \\
&\geq \int_{2k - M}^M e^s \mu(D_s) \, ds \geq \frac{\mu(D_{2k-M})}{2^d(M-k)^d} \int_{2k-M}^M (M-s)^d e^s \, ds \geq  \frac{\mu(D_{2k-M})e^M}{2^{d+3}(M-k)^d}.
\end{align*}
Thus $\mathbb{P}(X_1 \in D_{2k-M}) = \int_{D_{2k-M}} f_0 \leq 1/4$.  We deduce that
\[
\mathbb{P}\biggl(\frac{1}{n}\sum_{i=1}^n \phi(X_i) > k\biggr) \leq \mathbb{P}\biggl(\frac{1}{n}\sum_{i=1}^n \mathbbm{1}_{\{X_i \in D_{2k-M}\}} \geq \frac{1}{2}\biggr) \leq e^{-n/8},
\]
by Hoeffding's inequality.  The first Borel--Cantelli lemma then completes the proof of~(a).

(b) By the concavity of $\log \hat{f}_n$, it suffices to prove the conclusion of this part of the lemma when the infimum over $x \in \mathrm{conv} \ S$ is replaced with an infimum over $x \in S$.  Let $S$ be a compact subset of $\mathrm{int}(E)$ and let $\delta > 0$ be small enough that $S^\delta = \{x \in \mathbb{R}^d: \mathrm{dist}(x,S) \leq \delta\}$ is 
contained in $\mathrm{int}(E)$.  Consider the function $x_0 \mapsto \int_{B_{\delta/2}} f_0$, where $B_\delta$ again denotes the closed ball of radius $\delta$ centered at $x_0$.  This function is positive and continuous on $S^{\delta/2}$, so attains its (positive) infimum on this compact set.  Thus there exists $p > 0$ such that $\int_B f_0 \geq p$ for any Borel subset $B$ of $\mathbb{R}^d$ containing a ball of radius $\delta/2$ centered at a point in $S^{\delta/2}$.  

Now let $f$ be any log-concave density on $\mathbb{R}^d$, and let $c = 2\inf_{x \in S} f(x)$.  We show that if $c \geq 0$ is sufficiently small, then $f$ is not the maximum likelihood estimator for large $n$.  By Lemma~\ref{Lemma:Bound}(a), we may assume $\sup_{x \in \mathbb{R}^d} f(x) \leq C$.  Recall that the density $g$ was defined by $g(x) = e^{-\|x\|+b}$, where $b$ is a normalisation constant, and that $k = -\int \|x\| f_0(x) \, dx + b - 1$.  Suppose $c \in [0,C]$ is small enough that $\frac{p}{2}\log c + \bigl(1-\frac{p}{2}\bigr)\log C \leq k$.  Writing $B = \{x \in S^\delta: f(x) \leq c\}$, we note that $B$ contains a point $x_0 \in S$, and if $x \notin B$, then $2x_0 - x \in B$.  Thus $B$ contains a ball of radius $\delta/2$ centered at a point in $S^{\delta/2}$, so $\int_B f_0 \geq p$.  Thus, if $\phi = \log f$, then
\[
\mathbb{P}\biggl(\frac{1}{n}\sum_{i=1}^n \phi(X_i) > k\biggr) \leq \mathbb{P}\biggl(\frac{1}{n}\sum_{i=1}^n \mathbbm{1}_{\{X_i \in B\}} \leq \frac{p}{2}\biggr) \leq e^{-np^2/2},
\]
again by Hoeffding's inequality.  By the first Borel--Cantelli lemma, and arguing as in the proof of Lemma~\ref{Lemma:Bound}(a) above, we conclude that 
\[
\mathbb{P}\biggl(\frac{1}{n}\sum_{i=1}^n \log f(X_i) > \frac{1}{n}\sum_{i=1}^n \log g(X_i) \ \mathrm{i.o.}\biggr) = 0.
\]
\end{proof}

Our next theorem is the main result in this paper and establishes desirable performance properties of the log-concave maximum likelihood estimator.  We first recall that the Kullback--Leibler divergence of a density $f$ from $f_0$ is given by
\[
d_{KL}(f_0,f) = \int_{\mathbb{R}^d} f_0 \log \frac{f_0}{f}.
\]
Jensen's inequality shows that the Kullback--Leibler divergence is non-negative, and equal to zero if and only if $f=f_0$ (almost everywhere).  Thus in the case where $f_0$ is log-concave, Theorem~\ref{Thm:Main} below shows that the log-concave maximum likelihood estimator $\hat{f}_n$ is strongly consistent in certain exponentially weighted total variation metrics.  Convergence in exponentially weighted supremum norms also follows if $f_0$ is continuous.  The theorem strengthens known results even in the univariate case, which include Corollary~1 of \citet{PWM2007}, where it was proved that $\hat{f}_n$ is strongly consistent in Hellinger distance, and Corollary~4.2 of \citet{DumbgenRufibach2009}, where (weak) consistency of $\hat{f}_n$ in the unweighted total variation distance was established.  (The observation that the mode of convergence in the univariate consistency result of Corollary~4.2 of \citet{DumbgenRufibach2009} could be strengthened was also made independently at around the same time in \citet{SHD2009}.)

In the case where the model is misspecified, i.e. $f_0$ is not log-concave, we prove that the existence and uniqueness of a log-concave density $f^*$ that minimises the Kullback--Leibler divergence from $f_0$.  Moreover, we show that the log-concave maximum likelihood estimator $\hat{f}_n$ converges in the same senses as in the previous paragraph to $f^*$.  The natural practical interpretation is that provided $f_0$ is not too far from being log-concave, the estimator is still sensible.  

We write $\log_{+} x = \max(\log x,0)$ and recall that $E$ denotes the support of $f_0$.   
\begin{thm}
\label{Thm:Main}
Let $f_0$ be any density on $\mathbb{R}^d$ with $\int_{\mathbb{R}^d} \|x\| f_0(x) \, dx < \infty$, $\int_{\mathbb{R}^d} f_0 \log_{+} f_0 < \infty$ and $\mathrm{int}(E) \neq \emptyset$.  There exists a log-concave density $f^*$, unique almost everywhere, that minimises the Kullback--Leibler divergence of $f$ from $f_0$ over all log-concave densities $f$.  Taking $a_0 > 0$ and $b_0 \in \mathbb{R}$ such that $f^*(x) \leq e^{-a_0\|x\| + b_0}$, we have for any $a < a_0$ that 
\[
\int_{\mathbb{R}^d} e^{a\|x\|}|\hat{f}_n(x) - f^*(x)| \, dx \stackrel{a.s.}{\rightarrow} 0,
\]
and, if $f^*$ is continuous, $\sup_{x \in \mathbb{R}^d} e^{a\|x\|}|\hat{f}_n(x) - f^*(x)| \stackrel{a.s.}{\rightarrow} 0$.

\end{thm}
\textbf{Remark:} The conditions on the underlying density $f_0$ are very weak indeed.  The first condition asks for a finite mean, while the second is satisfied by any bounded density, as well as a wide class of unbounded densities.  The third condition is also very weak, but it may help to give an example where it fails: let $(q_n)$ be an enumeration of the rationals in $[0,1]$, and let $f_0 \propto \mathbbm{1}_E$, where $E = [0,1] \setminus \cup_{n=1}^\infty (q_n-\frac{1}{n^2},q_n+\frac{1}{n^2})$.  In this case $\mathrm{int}(E) = \emptyset$.   
\begin{proof}
By the two integrability conditions, the log-concave density $g(x) = e^{-\|x\|+b}$ , where $b$ is a normalisation constant, satisfies $d_{KL}(f_0,g) < \infty$.  We can therefore pick a minimising sequence of log-concave densities $(f_n)$ for the Kullback--Leibler divergence from $f_0$; in other words, the sequence $(f_n)$ satisfies
\[
d_{KL}(f_0,f_n) \rightarrow \inf_{f \in \mathcal{F}_0} d_{KL} (f_0,f),
\]
where $\mathcal{F}_0$ denotes the class of all log-concave densities.  A slightly simpler version of the argument given in the proof of Lemma~\ref{Lemma:Bound}(a) shows that there exists $C > 0$ such that $f_n \leq C$ for all $n$.  Similarly, a small modification to the argument in the proof of Lemma~\ref{Lemma:Bound}(b) shows that for every compact subset $S$ of $\mathrm{int}(E)$, there exists $c > 0$ such that
\[
\liminf_{n \rightarrow \infty} \inf_{x \in \mathrm{conv} \, S} f_n(x) \geq c.
\]
We claim that the sequence $(f_n)$ is tight (or more precisely that the sequence of probability measures corresponding to the sequence of densities is tight).  To see this, let $S$ be a compact subset in $\mathrm{int}(E)$, and choose $c > 0$ such that $\inf_{x \in S} f_n(x) \geq c$ for large $n$.  Without loss of generality we assume $0 \in S$ and $\mu(S) > 0$.  Now, for $R$ sufficiently large, we must have $\limsup_{n \rightarrow \infty} \sup_{\|x\| > R} f_n(x) \leq c/2$, as otherwise the Lebesgue measure of the convex level sets $\{x \in \mathbb{R}^d: f_n(x) > c/2\}$ would be too large for each $f_n$ to be a density.  It follows that there exist $a_0 > 0$ and $b_0 \in \mathbb{R}$ such that $\sup_{n \in \mathbb{N}} f_n(x) \leq e^{-a_0\|x\| + b_0}$, and tightness of the sequence follows.

Prohorov's theorem~\citep[][Theorem~5.1]{Billingsley1999} therefore guarantees the existence of a probability measure $\nu^*$ such that a subsequence $(f_{n_k})$ converges in distribution to $\nu^*$.  Now, given $\epsilon > 0$, choose $\delta = \epsilon/(2C)$.  If $A$ is a Borel subset of $\mathbb{R}^d$ with $\mu(A) \leq \delta$, then since Lebesgue measure is regular we can find an open set $A' \supseteq A$ such that $\mu(A') \leq 2\delta$.  Now 
\[
\nu^*(A) \leq \nu^*(A') \leq \liminf_{k \rightarrow \infty} \int_{A'} f_{n_k} \leq C\mu(A') \leq \epsilon.
\]
Thus $\nu^*$ is absolutely continuous with respect to $\mu$, and we may write $f^*$ for its density with respect to $\mu$.  By Proposition~\ref{Prop:Convergence}(a), $f^*$ is log-concave, and by Proposition~\ref{Prop:Convergence}(b), $f_{n_k} \rightarrow f^*$ almost everywhere.  Finally, by Fatou's lemma, we have
\begin{align*}
d_{KL}(f_0,f^*) = \int f_0 (\log f_0 - \log f^*) &\leq \liminf_{k \rightarrow \infty} \int f_0 (\log f_0 - \log f_{n_k}) \\
&= \inf_{f \in \mathcal{F}_0} d_{KL}(f_0,f).
\end{align*}
Thus $f^*$ does indeed minimise the Kullback--Leibler divergence from $f_0$ over the class of log-concave densities.

Suppose now that both $f_1^*$ and $f_2^*$ minimise the Kullback--Leibler divergence from $f_0$ over the class of log-concave densities.  Defining 
\[
f^* = \frac{(f_1^* f_2^*)^{1/2}}{\int (f_1^* f_2^*)^{1/2}},
\]
we see that $f^*$ is a log-concave density with 
\[
d_{KL}(f_0,f^*) = d_{KL}(f_0,f_1^*) + \log \int (f_1^* f_2^*)^{1/2} \leq d_{KL}(f_0,f_1^*),
\]
by the Cauchy--Schwarz inequality, with equality if and only if $f_1^* = f_2^*$, $\mu$-almost everywhere.  This proves the claimed uniqueness property of $f^*$.

Now, write $F_0$ for the distribution function corresponding to the density $f_0$ and $P_0$ for the distribution on $\mathbb{R}^d$ induced by $F_0$.  Similarly, write $\hat{\mathbb{F}}_n$ for the empirical distribution function of $X_1,\ldots,X_n$, and $\hat{\mathbb{P}}_n$ for the corresponding empirical measure.  By definition of $\hat{f}_n$, we have for any $b > 0$ that
\begin{align}
\label{Eq:BasicExp}
0 &\leq \int_{\mathbb{R}^d} \log (b+\hat{f}_n) \, d\hat{\mathbb{F}}_n - \int_{\mathbb{R}^d} \log f^* \, d\hat{\mathbb{F}}_n \nonumber \\
&= \int_{\mathbb{R}^d} \! \log (b+\hat{f}_n) \, d(\hat{\mathbb{F}}_n - F_0) + \int_{\mathbb{R}^d} \! \log \biggl(\frac{b+\hat{f}_n}{b+f^*}\biggr) \, dF_0 + \int_{\mathbb{R}^d} \! \log \biggl(\frac{b+f^*}{f^*}\biggr) \, dF_0 \nonumber \\
&\hspace{8cm}+ \int_{\mathbb{R}^d} \log f^* \, d(F_0 - \hat{\mathbb{F}}_n).
\end{align}
The idea of adding the small constant $b > 0$ in this calculation first appeared in \citet{PWM2007}.  We first derive an appropriate uniform law of large numbers to handle the first term on the right hand side of~(\ref{Eq:BasicExp}).  By Lemma~\ref{Lemma:Bound}(a), we may assume that $\hat{f}_n \leq C$.  Recall that $\mathcal{D}$ denotes the class of all Borel-measurable convex subsets of $\mathbb{R}^d$.  For any log-concave density $f$ with $f \leq C$, we have by Fubini's theorem that
\begin{align*}
\int_{\mathbb{R}^d} \log (b+f) \, d(\hat{\mathbb{F}}_n - F_0) &= \int_{\mathbb{R}^d} \log(1+f/b) \, d(\hat{\mathbb{F}}_n - F_0) \\
&= \int_{\mathbb{R}^d} \int_0^{\log(1+C/b)} \mathbbm{1}_{\{t \leq \log(1+f/b)\}} \, dt \, d(\hat{\mathbb{F}}_n - F_0) \\
&= \int_0^{\log(1+C/b)} (\hat{\mathbb{P}}_n - P_0)(\{x:f(x) \geq b(e^t-1)\}) \, dt \\
&\leq \log \Bigl(1+\frac{C}{b}\Bigr)\sup_{D \in \mathcal{D}} (\hat{\mathbb{P}}_n - P_0)(D) \stackrel{a.s.}{\rightarrow} 0
\end{align*}
as $n \rightarrow \infty$.  Hence
\[
\int_{\mathbb{R}^d} \! \log (b+\hat{f}_n) \, d(\hat{\mathbb{F}}_n - F_0) \stackrel{a.s.}{\rightarrow} 0
\]
as $n \rightarrow \infty$.

Combining this result with an application of the strong law of large numbers to the fourth term on the right-hand side of~(\ref{Eq:BasicExp}), we deduce that with probability one,
\[
\limsup_{n \rightarrow \infty} \int_{\mathbb{R}^d} \! \log \biggl(\frac{b+f^*}{b+\hat{f}_n}\biggr) \, dF_0 \leq  \int_{\mathbb{R}^d} \! \log \biggl(\frac{b+f^*}{f^*}\biggr) \, dF_0.
\]
It follows by the monotone convergence theorem that with probability one,
\[
\limsup_{b \searrow 0} \limsup_{n \rightarrow \infty} \int_{\mathbb{R}^d} \! \log \biggl(\frac{b+f^*}{b+\hat{f}_n}\biggr) \, dF_0 \leq 0.
\]
Lemma~\ref{Lemma:DiffConv} in the Appendix allows us to deduce from this that  $\int_{\mathbb{R}^d} |\hat{f}_n - f^*| \stackrel{a.s.}{\rightarrow} 0$, so the full result follows by Proposition~\ref{Prop:Convergence}.
\end{proof}

\section{Appendix}

Before proving Lemma~\ref{Lemma:ExpBound}, we first derive a basic property of a log-concave density $f$.  Recall that the epigraph of a concave function $\phi:\mathbb{R}^d \rightarrow [-\infty,\infty)$ is the set
\[
\{(x,\mu):x \in \mathbb{R}^d, \mu \in \mathbb{R}, \mu \leq \phi(x)\}.
\]
The closure of $\phi$, denoted $\mathrm{cl}(\phi)$, is the concave function whose epigraph is the closure in $\mathbb{R}^{d+1}$ of the epigraph of $\phi$.  The functions $\phi$ and $\mathrm{cl}(\phi)$ agree almost everywhere, and we say $\phi$ is closed if $\phi = \mathrm{cl}(\phi)$. 
\begin{lemma}
\label{Lemma:Bounded}
A log-concave density $f$ is bounded above and the version of $f$ that is closed attains its maximum.
\end{lemma}
\begin{proof}
Without loss of generality, we may assume $\log f$ is closed.  It has no directions of increase, because otherwise there would exist $\epsilon \in \mathbb{R}$ such that the set $\{x \in \mathbb{R}^d:\log f(x) \geq \epsilon\}$ were $d$-dimensional, convex and unbounded (so of infinite Lebesgue measure).  Theorem~27.2 of \citet{Rockafellar1997} therefore gives that $\log f$ attains its (finite) maximum.
\end{proof}

We can now prove Lemma~\ref{Lemma:ExpBound}.

\begin{prooftitle}{of Lemma~\ref{Lemma:ExpBound}}
Let $\phi = \log f$.  Without loss of generality, we may assume that $0 \in \mathrm{int}(\mathrm{dom} \ \phi)$.  Since $\phi(x) \rightarrow -\infty$ as $\|x\| \rightarrow \infty$, we can find $R > 0$ such that $\phi(x) - \phi(0) \leq -1$ for all $\|x\| > R$.  But, for each $z \in \mathbb{R}^d$ with $\|z\| = 1$, we have that $\frac{1}{c}\{\phi(cz) - \phi(0)\}$ is non-increasing in $c > 0$.  Thus, with $a = 1/R$, we have $\frac{1}{c}\{\phi(cz) - \phi(0)\} \leq -a$ for $c > R$.  Now 
\[
\frac{\phi(x)}{\|x\|} = \frac{\phi(x) - \phi(0)}{\|x\|} + \frac{\phi(0)}{\|x\|} \leq -a + \frac{\phi(0)}{\|x\|}
\]
for all $\|x\| > R$.  Since $\phi$ is bounded (by Lemma~\ref{Lemma:Bounded}), the result follows by choosing $b > \phi(0)$ sufficiently large.  \hfill $\Box$
\end{prooftitle}

The following lemma is used in the proof of Theorem~\ref{Thm:Main}.  The conclusion can be immediately strengthened using Proposition~\ref{Prop:Convergence}, and is stated in this way only for brevity.
\begin{lemma}
\label{Lemma:DiffConv}
Let $f_0$ be any density on $\mathbb{R}^d$ with $\int_{\mathbb{R}^d} \|x\| f_0(x) \, dx < \infty$, $\int_{\mathbb{R}^d} f_0 \log_{+} f_0 < \infty$ and $\mathrm{int}(E) \neq \emptyset$.  Let $f^*$ be a log-concave density that minimises the Kullback--Leibler divergence from $f_0$ over the class of log-concave densities.  If $(f_n)$ is a sequence of log-concave densities satisfying
\[
\limsup_{b \searrow 0} \limsup_{n \rightarrow \infty} \int_{\mathbb{R}^d} \log \biggl(\frac{b+f^*}{b+f_n}\biggr) \, dF_0 \leq 0,
\]
then $\int_{\mathbb{R}^d} |f_n - f^*| \rightarrow 0$ as $n \rightarrow \infty$.
\end{lemma}
\begin{proof}
Let $\Phi:\mathbb{R}^d \rightarrow \mathbb{R}$ be the Young function $\Phi(x) = (1+|x|) \log (1+|x|) - |x|$.  The Orlicz space $L^\Phi$ is the set of (equivalence classes of) measurable functions $f:\mathbb{R}^d \rightarrow \mathbb{R}$, whose Luxemburg norm $\|f\|_\Phi$, given by
\[
\|f\|_\Phi = \inf\Bigl\{\lambda > 0: \int \Phi(|f|/\lambda) \leq 1\Bigr\},
\]
is finite.  Let $\tilde{\Phi}(y) = e^{|y|} - |y| - 1$ denote the Young conjugate of $\Phi$, and let $\|\cdot\|_{\tilde{\Phi}}$ denote the corresponding Luxemburg norm on $L^{\tilde{\Phi}}$.  Then by \citet[Proposition~1, p.58]{RaoRen1991}, and the remark following it, for $f \in L^\Phi$ and $g \in L^{\tilde{\Phi}}$, we have
\[
\int |fg| \leq 2\|f\|_{\Phi}\|g\|_{\tilde{\Phi}}.
\]
Noting that $\|f_0\|_{\Phi} < \infty$, an immediate application of this formula yields that for any Borel subset $D$ of $\mathbb{R}^d$,
\begin{equation}
\label{Eq:Holder}
\int_D f_0 \leq \frac{2\|f_0\|_{\Phi}}{-\log \mu(D)}.
\end{equation}
Now let $f$ be a log-concave density with $\sup_{x \in \mathbb{R}^d} f(x) = C \equiv e^M$.  For large $M$, we have as in the proof of Lemma~\ref{Lemma:Bound}(a) that $\mu(\{x:f(x) \geq 1\}) \leq \frac{1}{8}M^d e^{-M} \leq e^{-M/2}$.  It follows that for any $b_0 \in (0,1)$ and $b < b_0$, we have
\begin{align*}
\int \log (b+f) \, dF_0 &\leq \int \log (b_0+f) \, dF_0 \leq \log (2b_0) + \int_{f \geq b_0} \log (2f) \, dF_0 \\
&\leq 2 \log 2 + \log b_0 + \int_{f \geq 1} \log f \, dF_0 \\
&\leq 2 \log 2 + \log b_0 + \frac{2M\|f_0\|_{\Phi}}{-\log \mu(\{x:f(x) \geq 1\})} \\
&\leq 2 \log 2 + \log b_0 + 4\|f_0\|_{\Phi} \rightarrow -\infty
\end{align*}   
as $b_0 \rightarrow 0$.  Here, the penultimate inequality uses~(\ref{Eq:Holder}).  We deduce that the sequence $(f_n)$ in the statement of the lemma satisfies the condition that there exists $C \geq 1$ such that
\[
\sup_{n \in \mathbb{N}} \sup_{x \in \mathbb{R}^d} f_n(x) \leq C.
\]
Now let $S$ be a compact subset of $\mathrm{int}(E)$.  Find $\delta > 0$ such that $S^\delta \subseteq \mathrm{int}(E)$ and, as in the proof of Lemma~\ref{Lemma:Bound}(b), find $p > 0$ such that $\int_B f_0 \geq p$ for all Borel subsets $B$ of $\mathbb{R}^d$ that contain a ball of radius $\delta/2$ centered at a point in $S^{\delta/2}$.  Let $f$ be any log-concave density on $\mathbb{R}^d$ with $\sup_{x \in \mathbb{R}^d} f(x) \leq C$, and write $c = 2\inf_{x \in S} f(x)$.  If $c \in [0,C]$ is sufficiently small, then we can find $b_0 > 0$ small enough that 
\[
p\log (b_0+c) + (1-p)\log(b_0+C) \leq \int \log f^* \, dF_0 - 1.
\]
Then writing $B = \{x:f(x) \leq c\}$, we have for all $b \in (0,b_0)$ that
\[
\int \log (b+f) \, dF_0 \leq \int_B \log (b_0 + c) \, dF_0 + \int_{B^c} \log (b_0 + C) \, dF_0 \leq \int \log f^* \, dF_0 - 1.
\]
We deduce that there exists $c > 0$ such that  
\begin{equation}
\label{Eq:Lowerbound}
\liminf_{n \rightarrow \infty} \inf_{x \in \mathrm{conv} \, S} f_n(x) \geq c.
\end{equation}
As in the proof of Theorem~\ref{Thm:Main}, we have from~(\ref{Eq:Lowerbound}) that the sequence $(f_n)$ is tight.  Thus if $(f_{n_k})$ is an arbitrary subsequence of $(f_n)$, then there exists a further subsequence $(f_{n_{k(l)}})$ and a log-concave density $f$ such that $\int |f_{n_{k(l)}} - f| \rightarrow 0$.  But then, by the dominated and monotone convergence theorems,
\begin{align*}
\limsup_{b \searrow 0} \limsup_{l \rightarrow \infty} \int_{\mathbb{R}^d} \log \biggl(\frac{b+f^*}{b+f_{n_{k(l)}}}\biggr) \, dF_0 &= \limsup_{b \searrow 0} \int_{\mathbb{R}^d} \log \biggl(\frac{b+f^*}{b+f}\biggr) \, dF_0 \\
&= \int_{\mathbb{R}^d} \log \frac{f^*}{f} \, dF_0 \geq 0,
\end{align*}
with equality if and only if $f=f^*$ almost everywhere.  By the hypothesis of the lemma, we must have $\int |f_{n_{k(l)}} - f^*| \rightarrow 0$.  Since every subsequence of $(f_n)$ has a further subsequence converging in total variation norm to $f^*$, we must have $\int |f_n - f^*| \rightarrow 0$.
\end{proof}

\textbf{Acknowledgements:} The second author is very grateful for helpful conversations with Richard Nickl and Michael Tehranchi.

\end{document}